\documentclass{article}
\usepackage[latin1]{inputenc}
\usepackage{amssymb, amsmath, amsthm}
\usepackage[a4paper]{geometry}
\usepackage{footnote}
\makesavenoteenv{tabular}
\usepackage{tabularx}
\usepackage{rotating}
\usepackage{color}

\theoremstyle{plain}

\newtheorem{definition}{Definition}[section]
\newtheorem{example}{Example}[section]
\newtheorem{lemma}{Lemma}[section]

\newtheorem{theorem}{Theorem}[section]

\begin{document}
%\pagenumbering{}
\begin{center}
\large{\textbf{The (logarithmic)  least squares optimality of
the arithmetic (geometric)  mean of weight vectors
calculated from all spanning trees for incomplete  additive (multiplicative)  pairwise comparison matrices}} \\[1cm]
{S\'andor Boz\'oki$^{1,2,3}$,}  {Vitaliy Tsyganok$^{4,5}$}   \\[1cm]
\end{center}
\footnotetext[1]{corresponding author}
\footnotetext[2]{Laboratory on Engineering and Management
Intelligence, Research Group of Operations Research and Decision
Systems, Institute for Computer Science and Control, Hungarian
Academy of Sciences; Mail: 1518 Budapest, P.O.~Box 63, Hungary. E-mail: bozoki.sandor@sztaki.mta.hu}
\footnotetext[3]{Department of Operations Research and Actuarial Sciences, Corvinus
University of Budapest, Hungary}
\footnotetext[4]{Laboratory for Decision Support Systems,
Institute for Information Recording of National Academy of Sciences of Ukraine; Mail:
2, Shpak str., Kyiv, 03113, Ukraine. E-mail: tsyganok@ipri.kiev.ua}
\footnotetext[5]{Department of Information Systems and Technologies,
Institute of Special Communication and Information Protection of
National Technical University of Ukraine
{\tiny{\raisebox{0.5mm}{$\ll$}}\footnotesize{Igor Sikorsky Kyiv Polytechnic Institute}\tiny{\raisebox{0.5mm}{$\gg$}}\footnotesize{.}
}}

\normalsize

\date{}

\begin{abstract}
Complete and incomplete  additive/multiplicative
pairwise comparison matrices are applied
in preference modelling, multi-attribute decision making and
ranking. The equivalence of two well known methods is proved in
this paper.
 The arithmetic (geometric)  mean of weight vectors,
calculated from all spanning trees, is proved to be optimal to the
 (logarithmic)
  least squares problem, not only for
complete, as it was recently shown in Lundy, M., Siraj, S.,
Greco, S. (2017): The mathematical equivalence of the ``spanning
tree'' and row geometric mean preference vectors and its
implications for preference analysis, European Journal of
Operational Research 257(1) 197--208, but for incomplete matrices
as well. Unlike the complete case, where an explicit formula, namely
 the row  arithmetic/geometric   mean of matrix elements,
 exists for the (logarithmic)
 least squares problem, the incomplete case requires a completely
dif{\kern0pt}ferent and new proof.
Finally, Kirchhof{\kern0pt}f's laws for the calculation of potentials in electric circuits is
connected to our results.


\end{abstract}

\section{Introduction}

Preference modelling is a family of qualitative and quantitative approaches in order to support decisions,
especially the choice of an alternative among a set of possible actions, or ranking them.
Many real decision problems involve multiple and often competing criteria
\cite{GrecoEhrgottFigueira2016}, therefore the weights of their
importance are also taken into account. Pairwise comparisons are applicable in both single and multiple
criteria decision making, as they divide complex problem into smaller tasks.

\subsection{Incomplete  multiplicative  pairwise comparison matrices}

Cardinal preferences of decision makers are often modelled and
calculated by pairwise comparison matrices \cite{Saaty1977}.
Questions 'How many times is a criterion more important than
another one?' or 'How many times is a given alternative better
than another one with respect to a f{\kern0pt}ixed criterion?' are
typical in multi-attribute decision problems. The numerical
answers are collected into a
 multiplicative  pairwise comparison matrix
$\mathbf{A}=[a_{ij}]_{i,j=1 \ldots n}$ fulf{\kern0pt}illing
reciprocity, i.e., $a_{ij} = 1/a_{ji}.$ A pairwise comparison
matrix can be complete, as in the Analytic Hierarchy Process (AHP)
\cite{Saaty1977}, or incomplete
\cite{BozokiFulopRonyai2010,CarmoneKaraZanakis1997,FedrizziGiove2007,Harker1987b,Kwiesielewicz1996,MengChen2015,LundySirajGreco2017,OlivaScalaSetolaDellOlmo2019,ShiraishiObataDaigo1998,SirajMikhailovKeane2012a,SirajMikhailovKeane2012b,TakedaYu1995,UrenaChiclanaMorente-MolineraHerrera-Viedma2015}.
A complete  multiplicative  pairwise comparison matrix $\mathbf{A}=[a_{ij}]$ is
called consistent if cardinal transitivity, i.e., $a_{ij} a_{jk} =
a_{ik}$ holds for all $i,j,k$. Otherwise, the matrix is
inconsistent, and several
 inconsistency indices have been proposed, see
 \cite{Brunelli2016,Brunelli2018,MengChen2015,Saaty1977}.

In this study
incomplete means 'not necessarily complete', in other words, the number of missing elements is allowed to be zero.

\begin{example} \label{6x6example} % \ref{6x6example}
Let $\mathbf{A}$ be a $6 \times 6$ incomplete  multiplicative  pairwise comparison matrix as follows:
\[
\mathbf{A} =
\begin{pmatrix}
     1    &   a_{12}     &            &     a_{14}   &    a_{15}    &     a_{16}        \\
  a_{21}  &     1        &    a_{23}  &              &              &                   \\
          &   a_{32}     &      1     &     a_{34}   &              &                   \\
  a_{41}  &              &    a_{43}  &      1       &    a_{45}    &                   \\
  a_{51}  &              &            &     a_{54}   &       1      &                   \\
  a_{61}  &              &            &              &              &         1
\end{pmatrix},
\]
 where $a_{ij} = 1/a_{ji}$ for all the known elements.
\end{example}

Incomplete pairwise comparison matrices can be applied not only in
the same multiple criteria decision situations in which the complete matrices arise
 (hundreds of case studies are listed in, e.g.,
\cite{Ho2008,SubramanianRamanathan2012,VaidyaKumar2006}),
 but also to larger decision and ranking problems.
Boz\'oki, Csat\'o and Temesi \cite{BozokiCsatoTemesi2016} proposed
a ranking method for top tennis players based on their pairwise
results, where incompleteness occurs in a natural way. Csat\'o
\cite{Csato2013} constructed a $149 \times 149$ incomplete
pairwise comparison matrix to rank the teams of the 39th Chess
Olympiad 2010.  Chao, Kou, Li and Peng \cite{ChaoKouLiPeng2018}
ranked 1544 Go players based on their matches played against each
other, which naturally formed an incomplete pairwise comparison
matrix. Duleba, Mishina and Shimazaki
\cite{DulebaMishinaShimazaki2012} applied small but incomplete
matrices in developing a decision model for urban bus
transportation supply. Ben\'{\i}tez, Delgado-Galv\'{a}n, Izquierdo
and P\'{e}rez-Garc\'{\i}a
\cite{BenitezDelgado-GalvanIzquierdoPerez-Garcia2015} calculated
the priorities from incomplete matrices in f{\kern0pt}inding the
best leakage control policy to minimize water loss. Krej\v{c}i
\cite[Chapter 5]{Krejci2018} presents an incomplete pairwise
comparison matrix based model for the evaluation of artistic
performance.

\subsection{The logarithmic least squares (LLS) problem  for multiplicative matrices }

The basic problem of f{\kern0pt}inding the \emph{best} weight vector
usually includes an additional information on how \emph{closeness} is
def{\kern0pt}ined or specif{\kern0pt}ied. The classical approaches
apply metrics based on least squares \cite{ChuKalabaSpingarn1979},
weighted least squares \cite{ChuKalabaSpingarn1979}, logarithmic
least squares
\cite{CrawfordWilliams1985,deGraan1980,deJong1984,Rabinowitz1976},
 just to name a few.
 Further weighting methods are discussed by Golany and Kress \cite{GolanyKress1993}
and by Choo and Wedley \cite{ChooWedley2004}.
Even the well-known eigenvector method \cite{Saaty1977} is proved to be a
distance minimizing method \cite{Fichtner1984,Fichtner1986}, although its
metric seems to be rather artif{\kern0pt}icial.\\

\begin{definition}
The Logarithmic Least Squares (LLS) problem \cite{Kwiesielewicz1996,TakedaYu1995} is def{\kern0pt}ined as follows:
\begin{align}
&\min \sum \limits_{\scriptsize{
             \begin{array}{c}
              i,j:  \\
              a_{ij} \text{ is known} \\
             \end{array}}}
\left[\log a_{ij}
-\log\left(\frac{w_{i}}{w_{j}}\right)\right]^2 \nonumber \\
 &                                                  \label{eq:IncompleteLLSMProblem-ObjFunction}  \\
%                                    %(\ref{eq:IncompleteLLSMProblem-ObjFunction})
&\text{subject to } \qquad w_{i} > 0, \qquad i=1,2,\dotsc,n. \nonumber
  %    \label{eq:IncompleteLLSMProblem-Positivity}
%                                    %(\ref{eq:IncompleteLLSMProblem-Positivity})
\end{align}
\end{definition}

Originally, the LLS problem was def{\kern0pt}ined for complete  multiplicative
pairwise comparison matrices, i.e., the sum in the objective
function is taken for all $i,j$ \cite{CrawfordWilliams1985,deGraan1980,deJong1984,Rabinowitz1976}.
In this special case, the LLS
optimal solution is unique and it can be explicitly computed by
taking  the row-wise geometric mean
\cite{CrawfordWilliams1985,deJong1984,Rabinowitz1976}.
Furthermore, in case of $3 \times 3$ complete pairwise comparison matrices, the eigenvector method
and the LLS method are equivalent, they result in the same weight vector
\cite{CrawfordWilliams1985}.
Several characterizations of the complete LLS weighting method  (or equivalently, the row geometric mean)
can be found in \cite{BarzilaiCookGolany1987,Csato2019,Fichtner1984,Fichtner1986}.  \\

The most common  scalings  are
$\sum\limits_{i=1}^{n}w_{i} = 1$ and $\prod\limits_{i=1}^{n}w_{i}
= 1.$  Scaling  $w_{1} = 1$ (called
\emph{ideal-mode} in Lundy, Siraj and Greco
\cite{LundySirajGreco2017}), can also be interpreted in the
following way: the f{\kern0pt}irst object (criterion, alternative)
is considered a reference point and all the others are expressed
according to it, similar to SMART \cite{Edwards1977},
if the f{\kern0pt}irst criterion is the least important one.

Given an (in)complete pairwise comparison matrix $\mathbf{A}$ of
size $n \times n$, an undirected graph $G(V,E)$ is def{\kern0pt}ined as follows: $G$ has $n$ nodes and
the edge between nodes $i$ and $j$ is drawn if and only if the matrix element
$a_{ij}$ is known. The graph of the incomplete pairwise comparison matrix in Example
\ref{6x6example} is given in Figure 1.

The graph-theoretic consideration makes it possible to represent the direct
comparison $a_{ij}$ between elements $i$ and $j$, as well as the indirect ones,
 e.g., via paths of two ($a_{ik},a_{kj}$), three ($a_{ik},a_{k{\ell}},a_{{\ell}j}$) or more edges
 \cite{Barzilai1997,BarzilaiCookGolany1987,Brugha2000,Gass1998,Harker1987b,HarkerVargas1987}.
 See also \cite[Subsection 2.2]{FedrizziGiove2007} as well as all references on spanning trees in subsection 1.4
 of this paper.

The following theorem provides a method for solving the
  LLS problem  (\ref{eq:IncompleteLLSMProblem-ObjFunction}).
\begin{theorem} \label{BozokiFulopRonyai2010theorem} % \ref{BozokiFulopRonyai2010theorem}
(Boz\'oki, F\"ul\"op, R\'onyai \cite[Section 4]{BozokiFulopRonyai2010})
Let $\mathbf{A}$ be an incomplete or complete  multiplicative  pairwise comparison matrix such that its
associated graph $G$ is connected. Then the optimal solution $\mathbf{w} = \exp \mathbf{y}$
of the logarithmic least squares problem
 (\ref{eq:IncompleteLLSMProblem-ObjFunction})
 is the unique solution of the following system of linear equations:
\begin{align}
\left( \mathbf{L}  \mathbf{y}          \right)_i  &=    \sum\limits_{k: e(i,k) \in E(G)} \log a_{ik}
\qquad \qquad \text{ for all } i=1,2,\ldots,n-1,n,  \label{equationLaplacian} \\
                                                                                    %  \ref{equationLaplacian}
                                    y_1  &= 0.        \label{normalizationy1=0}      % \ref{normalizationy1=0}
\end{align}
where $\mathbf{L}$ denotes the Laplacian matrix of $G$ ($\ell_{ii}$ is the degree of node $i$ and $\ell_{ij}=-1$
if nodes $i$ and $j$ are adjacent).
\end{theorem}
$\mathbf{L}$ has rank $n-1$.  Scaling
 (\ref{normalizationy1=0}), being equivalent to $w_1 = 1,$ plays a technical role only.
It can be replaced by, e.g., the commonly used $\prod_{i=1}^n w_i = 1 \, \, (\Leftrightarrow \sum_{i=1}^n y_i = 0).$
% The computational complexity of solving the system of $n$ linear
% equations (\ref{equationLaplacian})-(\ref{normalizationy1=0}) is at most $O(n^3)$ \cite[Chapter 1]{Watkins2002}.

\begin{example} \label{6x6exampleLaplacian} % \ref{6x6exampleLaplacian}
Let incomplete  multiplicative  pairwise comparison matrix $\mathbf{A}$ be the same as in Example \ref{6x6example}.
 Equations (\ref{equationLaplacian})
  for $i=1,2,\ldots,6$ form the following system of linear equations:
\[
\begin{pmatrix}
     4    &     -1       &      0     &     -1       &      -1      &        -1         \\
    -1    &      2       &     -1     &      0       &       0      &         0         \\
     0    &     -1       &      2     &     -1       &       0      &         0         \\
    -1    &      0       &     -1     &      3       &      -1      &         0         \\
    -1    &      0       &      0     &     -1       &       2      &         0         \\
    -1    &      0       &      0     &      0       &       0      &         1
\end{pmatrix}
\begin{pmatrix}
    y_1 (=0)    \\
    y_2    \\
    y_3    \\
    y_4    \\
    y_5    \\
    y_6
\end{pmatrix} =
\begin{pmatrix}
   \log a_{12}  + \log a_{14}  +\log a_{15} + \log  a_{16}    \\
   \log a_{21}  + \log a_{23}    \\
   \log a_{32}  + \log a_{34}    \\
   \log a_{41}  + \log a_{43}  +\log a_{45}   \\
   \log a_{51}  + \log a_{54}    \\
   \log a_{61}
\end{pmatrix},
\]
where the matrix of coef{\kern0pt}f{\kern0pt}icients above is the Laplacian matrix of the connected graph $G$ in Figure 1, that corresponds to
incomplete pairwise comparison matrix $\mathbf{A}$.
\unitlength 1mm
\begin{center}
\begin{picture}(100,60)
\put(25,12){\resizebox{50mm}{!}{\rotatebox{0}{
\includegraphics{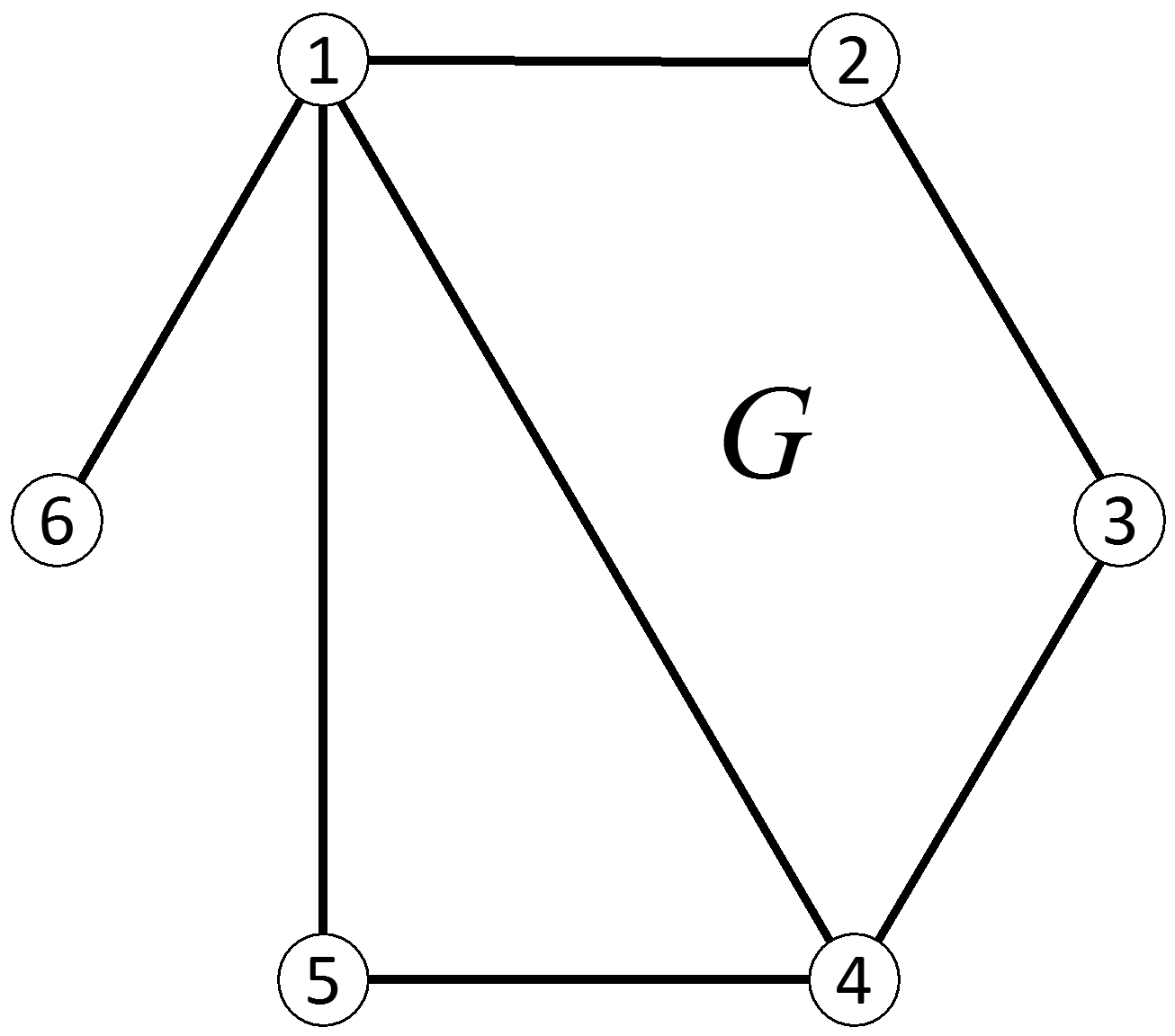}}}}
\put(22,0){\makebox{\emph{Figure 1. Graph $G$ of Example
\ref{6x6example}} }}
\end{picture}
\end{center}
\end{example}

\subsection{The least squares (LS) problem  for additive matrices}
Pairwise comparison matrices are relevant not only in multiplicative sense. An additive pairwise comparison matrix
 \cite{Barzilai1997,BarzilaiGolany1990}
$\mathbf{B}=[b_{ij}]_{i,j=1 \ldots n}$ fulf{\kern0pt}ils skew-symmetry, i.e., $b_{ij} = -b_{ji}.$
For every multiplicative pairwise comparison matrix $\mathbf{A}$,
$\mathbf{B}=\log(\mathbf{A})$ (elementwise) is an additive pairwise comparison matrix and vice versa
($\mathbf{A}=\exp(\mathbf{B})$).
The additive pairwise comparison matrix $\mathbf{B}$ is called
consistent if $b_{ij} + b_{jk} = b_{ik}$ holds for all $i,j,k$.
See \cite[Subsection 4.1.1]{Brunelli2015} for the applications of additive matrices in
  multi-criteria decision models like SMART \cite{Edwards1977} or
  REMBRANDT \cite[Chapter 12]{Lootsma1999},\cite{OlsonFliednerCurrie1995}.
Additive pairwise comparison matrices can also be incomplete, similar to the multiplicative ones.

\begin{example} \label{6x6exampleAdditive} % \ref{6x6exampleAdditive}
Recall Example \ref{6x6example}.
The incomplete additive  pairwise comparison matrix
$\mathbf{B}=\log(\mathbf{A})$ (elementwise, except for the missing ones) is as follows:
\[
\mathbf{B} =
\begin{pmatrix}
     0    &   b_{12}     &            &     b_{14}   &    b_{15}    &     b_{16}        \\
 -b_{12}  &     0        &    b_{23}  &              &              &                   \\
          &  -b_{23}     &      0     &     b_{34}   &              &                   \\
 -b_{14}  &              &   -b_{34}  &      0       &    b_{45}    &                   \\
 -b_{15}  &              &            &    -b_{45}   &       0      &                   \\
 -b_{16}  &              &            &              &              &         0
\end{pmatrix},
\]
\end{example}

The least squares (LS) minimization problem def{\kern0pt}ined for additive pairwise comparison matrices can
  be written as
\begin{align}
&\min \sum \limits_{\scriptsize{
             \begin{array}{c}
              i,j:  \\
              b_{ij} \text{ is known} \\
             \end{array}}}
\left(b_{ij} - y_{i} + y_{j} \right)^2  \nonumber \\
  &                                                   \label{eq:IncompleteLSMProblem-ObjFunction}  \\
%                                    %(\ref{eq:IncompleteLSMProblem-ObjFunction})
&\text{subject to } \qquad y_{1} = 0. \nonumber
 %    \label{eq:IncompleteLSMProblem-Normalization}
%                                    %(\ref{eq:IncompleteLSMProblem-Normalization})
\end{align}

The least squares minimization problem for additive matrices
 (4)  %(\ref{eq:IncompleteLSMProblem-ObjFunction})
is widely applied in multi-criteria decision making  and preference modelling,
see \cite{Barzilai1997,Barzilai1998,BarzilaiGolany1990,Lootsma1999}.

The LS problem (4)  % (\ref{eq:IncompleteLSMProblem-ObjFunction})
 can be traced back to Thurstone \cite{Thurstone1927}  and Horst \cite{Horst1932}.

LS is
among the scoring models discussed by Chebotarev and Shamis \cite[Section 8.1]{ChebotarevShamis1999}, or
in the context of preference graphs by \v{C}aklovi\'{c} and Kurdija \cite[Section 2]{CaklovicKurdija2017}.

Note that Theorem \ref{BozokiFulopRonyai2010theorem} applies to the
LS problem (4)  % \label{eq:IncompleteLSMProblem-ObjFunction},
 too, with $\mathbf{A}=\exp(\mathbf{B})$.

 Rewording the def{\kern0pt}inition of consistency,
 $a_{ij} a_{jk} = a_{ik} \Leftrightarrow  a_{ij} a_{jk} a_{ki} = 1 $ (multiplicative)
 and $b_{ij} + b_{jk} = b_{ik}  \Leftrightarrow b_{ij} + b_{jk} + b_{ki} = 0 $ (additive),
 require that the product/sum of matrix elements in any 3-cycle must be 1/0.
 This leads to the more general def{\kern0pt}inition of consistency that can be applied to both complete
 and incomplete pairwise comparison matrices.
\begin{definition}
A multiplicative/additive (in)complete pairwise comparison matrix $\mathbf{A}/\mathbf{B}$ is called
consistent, if
$a_{i_1 i_2} \cdot a_{i_2 i_3} \cdot \ldots \cdot a_{i_k i_1} = 1 $ /
$b_{i_1 i_2} +  b_{i_2 i_3}  +  b_{i_k i_1} = 0 $
for any cycle $i_1,i_2,\ldots i_k,i_1$ in the graph of the matrix.
\end{definition}
Note that this def{\kern0pt}inition is equivalent to that the incomplete matrix can be (fully) completed
such that the complete matrix is consistent. Furthermore this completion is unique if and only if
the graph is connected. It follows from the def{\kern0pt}inition that an incomplete matrix with an acyclic graph
 (a tree or a disjoint union of trees) is consistent.
Consistency is also equivalent to that the optimum value of the
 logarithmic least squares (\ref{eq:IncompleteLLSMProblem-ObjFunction}) /
least squares (4)  % (\ref{eq:IncompleteLSMProblem-ObjFunction})
 problem is 0. Again, the optimal solution is unique up to scaling if and only if the graph is connected. \\

 The close relation of Def{\kern0pt}inition 1.2 to
 Kirchhof{\kern0pt}f's Voltage Law (the signed sum of the potential dif{\kern0pt}ferences around any closed loop is zero)
 is recalled in Section 3.

\subsection{Aggregations of weight vectors calculated from all spanning trees}

The spanning tree approach by Tsyganok \cite{Tsyganok2000,Tsyganok2010}
does not assume any distance function or measure of closeness. The basic idea
is that the set of pairwise comparisons is considered as the union of minimal, connected
subsets, or, in graph-theoretic terms, spanning trees.
Let $S$ denote the number of all spanning trees of graph $G$.
Every spanning tree determines a unique
weight vector f{\kern0pt}itting on the corresponding subset of matrix elements perfectly,
   as the incomplete pairwise comparison matrix associated to a spanning tree
  is consistent according to Def{\kern0pt}inition 1.2.
Given a spanning tree, the calculation of its associated weight vector requires $O(n)$ steps.

The number of spanning trees can be very large.
In the special case of complete pairwise comparison matrices, the number of all spanning trees is
$S= n^{n-2}$ by Cayley's theorem.
Another extremal case is when the graph of the incomplete pairwise comparison
matrix is itself a tree ($S=1$).

The most natural candidates for the aggregation of weight vectors calculated from all spanning trees
are the arithmetic
 \cite{SirajMikhailovKeane2012a,SirajMikhailovKeane2012b,Tsyganok2000,Tsyganok2010}
and the geometric means \cite{LundySirajGreco2017,TsyganokKadenkoAndriichuk2015}.

The following theorem connects two weighting methods.
\begin{theorem}  \label{LundySirajGreco2017theorem} % \ref{LundySirajGreco2017theorem}
(Lundy, Siraj and Greco \cite{LundySirajGreco2017})
The geometric mean of weight vectors calculated from all spanning trees
is logarithmic least squares optimal in case of complete  multiplicative  pairwise comparison matrices.
\end{theorem}

 The rest of the paper is organized as follows.
 The proof of Theorem
\ref{LundySirajGreco2017theorem} is based on that
 an explicit formula (row geometric mean of matrix elements) exists for the
 complete   LLS problem
\cite{CrawfordWilliams1985,deJong1984}. As the incomplete LLS problem
does not have such a closed form solution, only an implicit one according to equations (\ref{equationLaplacian}),
a new and essentially dif{\kern0pt}ferent approach is needed to extend the theorem to the case of missing elements.
This theorem, the main result of the paper, stating that
the geometric mean of weight vectors calculated from all spanning trees
is logarithmic least squares optimal in both cases of incomplete and complete  multiplicative  pairwise comparison matrices,
is given in Section 2.   Equivalently, the arithmetic mean of weight vectors calculated from all spanning trees
is least squares optimal for additive pairwise comparison matrices.
Section 3 shows that spanning trees appear in a natural way in electric circuits,
and the calculation of potentials with Kirchhof{\kern0pt}f's Rules is directly related to the
 least squares problem  written for additive matrices.
Section 4 concludes with computational complexity and open questions.

\section{Main result: the  arithmetic (geometric)  mean of weight vectors calculated from all spanning trees is (logarithmic)  least squares optimal}

\begin{theorem} \label{maintheorem} % \ref{maintheorem}
 \textbf{(multiplicative)}  Let $\mathbf{A}$ be an incomplete or complete  multiplicative  pairwise comparison matrix such that its
associated graph is connected.
Then the optimal solution of the logarithmic least squares problem
 (\ref{eq:IncompleteLLSMProblem-ObjFunction})
is equal, up to a scalar multiplier, to the geometric mean of weight vectors
calculated from all spanning trees.
\end{theorem}

Before proving,  let us rephrase Theorem \ref{maintheorem} with the elementwise logarithm of an incomplete
or complete  multiplicative  pairwise comparison matrix, which is a(n incomplete)  additive  (skew symmetric) matrix, let us denote it by
$\mathbf{B}$. An undirected graph $G$ is associated to $\mathbf{B}$ as follows: it
has $n$ nodes and the edge between nodes $i$ and $j$ is drawn if and only if the matrix element
$b_{ij}$ is given. Let $T^1, T^2, \ldots, T^s, \ldots, T^S$ denote the spanning trees of $G$.
Let $\mathbf{y}^s \in \mathbb{R}^n, \, s=1,2,\ldots,S$, be the weight vector
calculated from spanning tree $T^s$ and scaled by $y_{1} = 0$.

\begin{theorem} \label{LSMtheorem} % \ref{LSMtheorem}
 \textbf{(additive)}
Let $\mathbf{B}$ be an incomplete or complete  additive  (skew symmetric) matrix such that its
associated graph is connected. Then the optimal solution of the least squares problem
  (4) is equal to the arithmetic mean of weight vectors calculated from all spanning trees, each one
scaled by $y^s_{1} = 0$.
\end{theorem}

\begin{proof}
Let $G$ be the connected graph associated with the (in)complete  multiplicative
pairwise comparison matrix $\mathbf{A}$ and let $E(G)$ denote the
set of edges. The edge between nodes $i$ and $j$ is denoted by
$e(i,j)$. The Laplacian matrix of graph $G$ is denoted by
$\mathbf{L}$. Let $T^1, T^2, \ldots, T^s, \ldots, T^S$ denote the
spanning trees of $G$, where $S$ denotes the number of spanning
trees. $E(T^s)$ denotes the set of edges in $T^s$. Hereafter,
upper index $s$ is also used for indexing a weight vector or a
pairwise comparison matrix, associated to spanning tree $T^s$. Let
$\mathbf{w}^s, s=1,2,\ldots,S,$ denote the weight vector
calculated from spanning tree $T^s$. Weight vector $\mathbf{w}^s$
is unique up to a scalar multiplier. For sake of simplicity we can
assume that $w_1^s = 1$, but other ways of  scaling,
 e.g., $\prod w_i = 1$ can also be chosen. Let
$\mathbf{y}^s := \log \mathbf{w}^s, \, s=1,2,\ldots,S$, where the
logarithm is taken element-wise. Let $\mathbf{w}^{LLS}$ denote the
optimal solution to the LLS problem (scaled by $w_1^{LLS} = 1$) and $\mathbf{y}^{LS} := \log
\mathbf{w}^{LLS}$.  The formal statement of Theorem
\ref{maintheorem} is that
\[
{w}^{LLS}_i = \sqrt[\mbox{\begin{normalsize}\emph{S}\end{normalsize}}]{ \prod\limits_{s=1}^{S}{w}^s_i }, \qquad i=1,2,\ldots,n,
\]
that is, by taking the logarithm, equivalent to
\[
\mathbf{y}^{LS} = \frac{1}{S} \sum\limits_{s=1}^{S}\mathbf{y}^s,
\]
(which is the statement of Theorem \ref{LSMtheorem})      that we shall prove.
By Theorem \ref{BozokiFulopRonyai2010theorem},
\begin{equation*}
\left( \mathbf{L} \mathbf{y}^{LS} \right)_i = \sum\limits_{k: e(i,k) \in E(G)} b_{ik}
\qquad \qquad \text{ for all } i=1,2,\ldots,n,
\end{equation*}
where $b_{ik} = \log a_{ik}$ for all $(i,k) \in E(G).$
 Since graph $G$ is connected, vector
$\mathbf{y}^{LS}$ is unique with the scaling $y_1^{LS} = 0$.

It is  therefore
suf{\kern0pt}f{\kern0pt}icient to show that
\begin{equation} \label{equation:goal} % \ref{equation:goal}
\left( \mathbf{L} \frac{1}{S} \sum\limits_{s=1}^{S}\mathbf{y}^s \right)_i = \sum\limits_{k: e(i,k) \in E(G)} b_{ik}
\qquad \qquad \text{ for all } i=1,2,\ldots,n.
\end{equation}

Observe that the Laplacian matrices of any two spanning trees are dif{\kern0pt}ferent, therefore
'intermediate' incomplete  multiplicative  pairwise comparison matrices are needed.
Consider an arbitrary spanning tree $T^s$. Then $\frac{w^s_i}{w^s_j} = a_{ij}$ for all
$e(i,j) \in E(T^s)$. Introduce the incomplete  multiplicative
  pairwise comparison matrix $\mathbf{A}^s$
 by $a^s_{ij} := a_{ij}$ for all $e(i,j) \in E(T^s)$ and
$a^s_{ij} := \frac{w^s_i}{w^s_j}$ for all $e(i,j) \in E(G) \backslash E(T^s)$.
 The incomplete  multiplicative
 pairwise comparison matrix $\mathbf{A}^s$ is consistent
 according to Def{\kern0pt}inition 1.2 for all $s=1,2,\ldots,S.$
Again, $b^s_{ij} := \log a^s_{ij} (= y^s_i - y^s_j)$.
Now the Laplacian matrices of $\mathbf{A}$ and $\mathbf{A}^s$ are the same ($\mathbf{L}$).
Since the weight vector $\mathbf{w}^s$ is generated by the matrix elements
belonging to spanning tree $T^s$, it is also the optimal solution of the
 LLS problem regarding
$\mathbf{A}^s$  (furthermore, the optimum value is
zero, because $a^s_{ij} = \frac{w^s_i}{w^s_j}$ for all $e(i,j) \in
E(G)$).  Equivalently, the following system of linear
equations holds.
\begin{equation}
\left( \mathbf{L} \mathbf{y}^s \right)_i=
\sum\limits_{k: e(i,k) \in E(T^s)} b_{ik} + \sum\limits_{k: e(i,k) \in E(G) \backslash E(T^s)} b^s_{ik}
\qquad \qquad
\text{ for all } i=1,2,\ldots,n. \label{equation:Ly^s} % \ref{equation:Ly^s}
\end{equation}

\begin{lemma} \label{lemma:average} % \ref{lemma:average}
\begin{equation}
\sum\limits_{s=1}^{S}
\left(\sum\limits_{k: e(i,k) \in E(T^s)} b_{ik} + \sum\limits_{k: e(i,k) \in E(G) \backslash E(T^s)} b^s_{ik} \right) =
S \sum\limits_{k: e(i,k) \in E(G)} b_{ik}. \label{equationlemma:average} % \ref{equationlemma:average}
\end{equation}
\end{lemma}
\begin{proof} % of Lemma \ref{lemma:average}.
Let $i$ be f{\kern0pt}ixed arbitrarily and consider node $i$ in all spanning trees.
There is nothing to do with edges $e(i,k) \in E(T^s)$.
Since $T^s$ is a spanning tree,
for every edge $e(i,k) \in E(G) \backslash E(T^s)$ there exists a unique path \\
$ P= \{ e(i,k_1), e(k_1,k_2), \ldots , e(k_{\ell},k) \} \subseteq E(T^s)$.
$P \cup e(i,k)$ is a cycle and
\begin{equation} \label{equation:bsik} % \ref{equation:bsik}
b^s_{i k} = b_{i k_1} + b_{k_1 k_2} + \ldots + b_{k_{\ell} k}.
\end{equation}
Consider the following spanning tree:
$T^{s_{i,k,k_1}^{\prime}} := (T^s \backslash e(i,k_1) ) \cup e(i,k)$
 as in Figure 2.
% \newpage
\unitlength 1mm
\begin{center}
\begin{picture}(100,125)
\put(-26,55){\resizebox{150mm}{!}{\rotatebox{0}{
\includegraphics{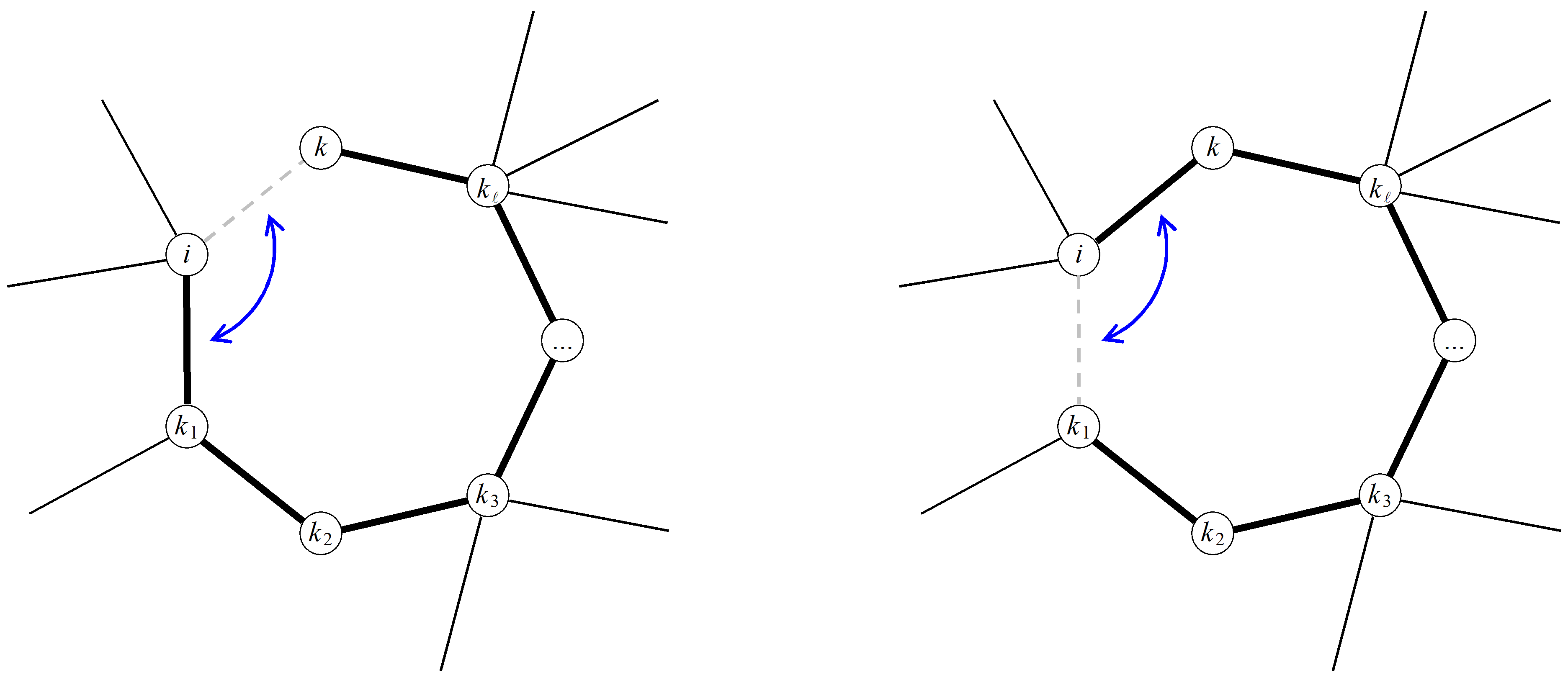}}}}  %
\put(10,85){\makebox{$T^s$}}
\put(90,85){\makebox{$T^{s_{i,k,k_1}^{\prime}}$}}
\put(-13,45){\makebox{Figure 2. The replacement of edge $e(i,k_1)$
in spanning tree $T^s$ by edge $e(i,k)$}}
\put(25,40){\makebox{results in spanning tree
$T^{s_{i,k,k_1}^{\prime}}$.}}
\end{picture}
\end{center}

\noindent
Spanning trees $T^s$ and $T^{s_{i,k,k_1}^{\prime}}$ dif{\kern0pt}fer in one edge only and
\begin{equation} \label{equation:bsprimeik1} % \ref{equation:bsprimeik1}
b^{s_{i,k,k_1}^{\prime}}_{i k_1} = b_{i k} + b_{k k_{\ell}} + \ldots + b_{k_2 k_1}.
\end{equation}
Adding up equations (\ref{equation:bsik}) and (\ref{equation:bsprimeik1})
results in
\begin{equation} \label{equation:bsik+bsprimeik1} % \ref{equation:bsik+bsprimeik1}
b^s_{i k} + b^{s_{i,k,k_1}^{\prime}}_{i k_1} = b_{i k} + b_{i k_1},
\end{equation}
all intermediate terms vanish due to the reciprocal property of pairwise comparison matrices.
Now let us continue this process and go through all edges $e(i,k) \in E(G) \backslash E(T^s)$
for all $k$ and $s$. The remarkable symmetry of the set of all spanning trees implies that
every edge occurs in exactly one pair.
Summing all these equations like (\ref{equation:bsik+bsprimeik1}), the statement
 of Lemma \ref{lemma:average} follows.
\end{proof}

 We can now complete the proof of Theorem
\ref{maintheorem}: add up equations in Eq.~(\ref{equation:Ly^s})
for all $s=1,2,\ldots,S$, then divide by $S$, then the left hand
side becomes the left hand side of Eq.~(\ref{equation:goal}). The
identity of the right hand sides follows from Lemma
\ref{lemma:average}, therefore Eq.~(\ref{equation:goal}) is
proved. It implies $\mathbf{y}^{LS} = \frac{1}{S}
\sum\limits_{s=1}^{S}\mathbf{y}^s,$ and, equivalently,
${w}^{LLS}_i = \sqrt[\mbox{\begin{small}\emph{S}\end{small}}]{
\prod\limits_{s=1}^{S}{w}^s_i }, \, i=1,2,\ldots,n,$ which is the
statement of Theorem \ref{maintheorem}.
\end{proof}

\textbf{Remark.} Complete pairwise comparison matrices ($S=n^{n-2}$)
are included in Theorems \ref{maintheorem} and  as a special case.
The proof of Theorem \ref{maintheorem} can also be considered as a
second and shorter proof of Theorem \ref{LundySirajGreco2017theorem}.

%\newpage
\begin{example} \label{lemmaexample} % \ref{lemmaexample}
(An illustration of the proof of Theorem \ref{maintheorem})

Let incomplete  multiplicative   pairwise comparison matrix $\mathbf{A}$ be the same as in Example \ref{6x6example}.
The associated graph $G$ and its  $(S=11)$
spanning trees $T^1, T^2, \ldots, T^{11}$ are shown in Figure 3.
Consider spanning tree $T^1$ having edges $e(1,5),e(1,6),e(2,3),e(3,4),e(4,5),e(5,6)$.
Simple calculation results in its weight vector
\[
\mathbf{w}^1 =
\begin{pmatrix}
  1   \\
  a_{23} a_{34} a_{45} / a_{15} \\
  a_{34} a_{45} / a_{15}  \\
  a_{45} / a_{15}  \\
  1/a_{15} \\
  1/a_{16}
\end{pmatrix}.
\]
Ratios $\frac{w^1_{i}}{w^1_{j}} = a_{ij}$ for all $i,j$ such that $e(i,j) \in E(T^1)$.
In order to write the incomplete  multiplicative   pairwise
 comparison matrix $\mathbf{A}^1$, we need edges
 $ e(1,2), e(1,4) \in E(G) \backslash E(T^1)$ and the corresponding equations
 $ a^1_{12} := \frac{w^1_{1}}{w^1_{2}}$ and $ a^1_{14} := \frac{w^1_{1}}{w^1_{4}}$. Then
\[
\mathbf{A}^1 =
\begin{pmatrix}
     1        & a_{15}/(a_{23} a_{34} a_{45})
                         &            & a_{15}/a_{45} &    a_{15}    &     a_{16}        \\
a_{23} a_{34} a_{45} / a_{15}
              &     1        &    a_{23}  &               &              &                   \\
              &   a_{32}     &      1     &      a_{34}   &              &                   \\
a_{45}/a_{15} &              &    a_{43}  &       1       &    a_{45}    &                   \\
  a_{51}      &              &            &      a_{54}   &       1      &                   \\
  a_{61}      &              &            &               &              &         1
\end{pmatrix}.
\]

Then equations (\ref{equation:Ly^s}) for $s=1$ are as follows:
\[
%\hspace*{-4mm}
\begin{pmatrix}
     4    &     -1       &      0     &     -1       &      -1      &        -1         \\
    -1    &      2       &     -1     &      0       &       0      &         0         \\
     0    &     -1       &      2     &     -1       &       0      &         0         \\
    -1    &      0       &     -1     &      3       &      -1      &         0         \\
    -1    &      0       &      0     &     -1       &       2      &         0         \\
    -1    &      0       &      0     &      0       &       0      &         1
\end{pmatrix}
\begin{pmatrix}
  0   \\
  b_{23} + b_{34} + b_{45} - b_{15} \\
  b_{34} + b_{45} - b_{15}  \\
  b_{45} - b_{15}  \\
  -b_{15} \\
  -b_{16}
\end{pmatrix} =
\begin{pmatrix}
   b_{15} + b_{16}      \\
   b_{23}               \\
  -b_{23} + b_{34}      \\
  -b_{34} + b_{45}      \\
  -b_{15} + b_{45}      \\
  -b_{16}
\end{pmatrix} +
\begin{pmatrix}
   b^1_{12} + b^1_{14}    \\
   b^1_{21}               \\
   0                      \\
   b^1_{41}               \\
   0                      \\
   0
\end{pmatrix},
\]
where
$ b^1_{12} = b_{15} - b_{23} - b_{34} - b_{45}$,
$ b^1_{21} = -b^1_{12} = - b_{15} + b_{23} + b_{34} + b_{45}$
and
$ b^1_{41} = b_{45} - b_{15} $.

We have that weight vector $\mathbf{w}^1$ is the unique solution
 to both of the LLS problems
\begin{align*}
&\min \sum \limits_{\scriptsize{
             \begin{array}{c}
              i,j:  \\
              e(i,j) \in E(T^1) \\
             \end{array}}}
\left[\log a_{ij}
-\log\left(\frac{w_{i}}{w_{j}}\right)\right]^2  \\
&\text{subject to } \qquad w_{i} > 0, \qquad i=1,2,\dotsc,6, \\
&\qquad \qquad \, \, \qquad w_{1} = 1,
\end{align*}
and
\begin{align*}
&\min \sum \limits_{\scriptsize{
             \begin{array}{c}
              i,j:  \\
              e(i,j) \in E(G) \\
             \end{array}}}
\left[\log a^1_{ij}
-\log\left(\frac{w_{i}}{w_{j}}\right)\right]^2  \\
&\text{subject to } \qquad w_{i} > 0, \qquad i=1,2,\dotsc,6, \\
&\qquad \qquad \, \, \qquad w_{1} = 1,
\end{align*}
and the optimum values are zeros in both cases.

Now let us focus on Lemma \ref{lemma:average} with
node $i=1$.  Edges adjacent to node
 1 are missing 12 times (and they are not missing 32
times) in the whole set of spanning trees, hence we can identify 6
pairs. They induce 6 pairs of equations, that are labelled in
Figure 3. In tree $T^1$,
\begin{equation} \label{11}
b_{12}^{1} = b_{15} + b_{54} + b_{43} + b_{32}.
\end{equation}
Note that equation (11), as well as the forthcoming ones, is
labelled on the corresponding edges in Figure 3.
Now $s=1, k=2, k_1=5$ and ${s_{1,2,5}^{\prime}} = 4$, because the replacement of edge $e(1,5)$ in tree $T^1$
by edge $e(1,2)$ results in tree $T^4$. Here
\begin{equation} \label{12}
b_{15}^{4} = b_{12} + b_{23} + b_{34} + b_{45}.
\end{equation}

The sum of equations (11) and (12) conf{\kern0pt}irms (\ref{equation:bsik+bsprimeik1}).

Let us continue by edge $e(1,4)$ in tree $T^1$.
\begin{eqnarray}
b_{14}^{1} = b_{15} + b_{54}, \label{13} \\
b_{15}^{2} = b_{14} + b_{45}. \label{14}
\end{eqnarray}

The remaining four pairs of edges and their equations are listed below.
\begin{eqnarray}
b_{12}^{2} = b_{14} + b_{43} + b_{32}, \label{15} \\
b_{14}^{4} = b_{12} + b_{23} + b_{34}, \label{16}
\end{eqnarray}
\begin{eqnarray}
b_{12}^{3} = b_{14} + b_{43} + b_{32}, \label{17} \\
b_{14}^{7} = b_{12} + b_{23} + b_{34}, \label{18}
\end{eqnarray}
\begin{eqnarray}
b_{14}^{5} = b_{15} + b_{54},          \label{19} \\
b_{15}^{8} = b_{14} + b_{45},          \label{20}
\end{eqnarray}
\begin{eqnarray}
b_{14}^{6} = b_{15} + b_{54},          \label{21} \\
b_{15}^{9} = b_{14} + b_{45}.          \label{22}
\end{eqnarray}

Lemma \ref{lemma:average} is now conf{\kern0pt}irmed for $i=1$:
\begin{equation*}
\sum\limits_{s=1}^{11}
\left(\sum\limits_{k: e(1,k) \in E(T^s)} b_{1k} + \sum\limits_{k: e(1,k) \in E(G) \backslash E(T^s)} b^s_{1k} \right) =
11 \sum\limits_{k: e(1,k) \in E(G)} b_{1k} = 11 (b_{12}+b_{14}+b_{15}+b_{16}).
\end{equation*}

%\newpage
Let us move to node $2$. Three pairs of equations can be obtained:
\begin{eqnarray}
b_{21}^{1} = b_{23} + b_{34} + b_{45} + b_{51},      \label{23} \\
b_{23}^{5} = b_{21} + b_{15} + b_{54} + b_{43},      \label{24}
\end{eqnarray}
\begin{eqnarray}
b_{21}^{2} = b_{23} + b_{34} + b_{41},               \label{25} \\
b_{23}^{8} = b_{21} + b_{14} + b_{43},               \label{26}
\end{eqnarray}
\begin{eqnarray}
b_{21}^{3} = b_{23} + b_{34} + b_{41},               \label{27} \\
b_{23}^{10}= b_{21} + b_{14} + b_{43}.               \label{28}
\end{eqnarray}

Lemma \ref{lemma:average} is now conf{\kern0pt}irmed for $i=2$:
\begin{equation*}
\sum\limits_{s=1}^{11}
\left(\sum\limits_{k: e(2,k) \in E(T^s)} b_{2k} + \sum\limits_{k: e(2,k) \in E(G) \backslash E(T^s)} b^s_{2k} \right) =
11 \sum\limits_{k: e(2,k) \in E(G)} b_{2k} = 11 (b_{21}+b_{23} ).
\end{equation*}

Cases related to the remaining nodes can be treated likewise.

% \newpage
\unitlength 1mm
\begin{center}
\begin{picture}(100,190)
\put(-26,10){\resizebox{150mm}{!}{\rotatebox{0}{
\includegraphics{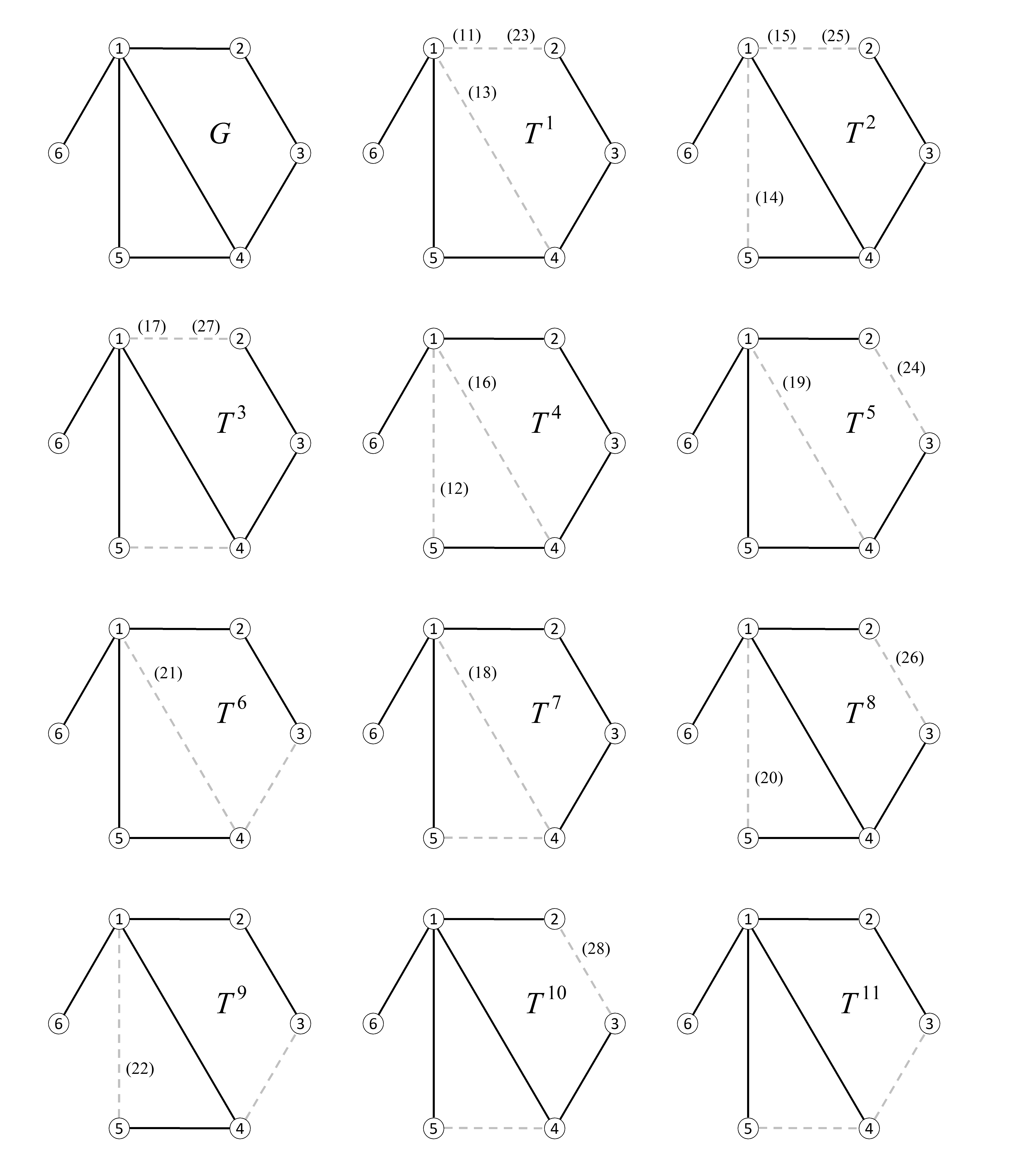}}}}
\put(-8,0){\makebox{\emph{Figure 3. Graph $G$ of Example
\ref{lemmaexample} and its spanning trees $T^1, T^2, \ldots,
T^{11}$} }}
\end{picture}
\end{center}

\end{example}

\newpage
\section{Electric circuits and potentials}

The least squares problem for additive matrices (4) % (\ref{eq:IncompleteLSMProblem-ObjFunction})
occurs in a natural way not only in decision theory, but in physics as well.
Energy minimization and potentials in electric circuits are discussed in this section, namely,
 the least squares problem (4) and Theorem \ref{LSMtheorem} are  illustrated by an example.

\begin{example} \label{circuitexample} % \ref{circuitexample}
Consider the following electric circuit on four nodes.
\unitlength 1mm
\begin{center}
\begin{picture}(100,40)
\put(30,3){\resizebox{50mm}{!}{\rotatebox{0}{
\includegraphics{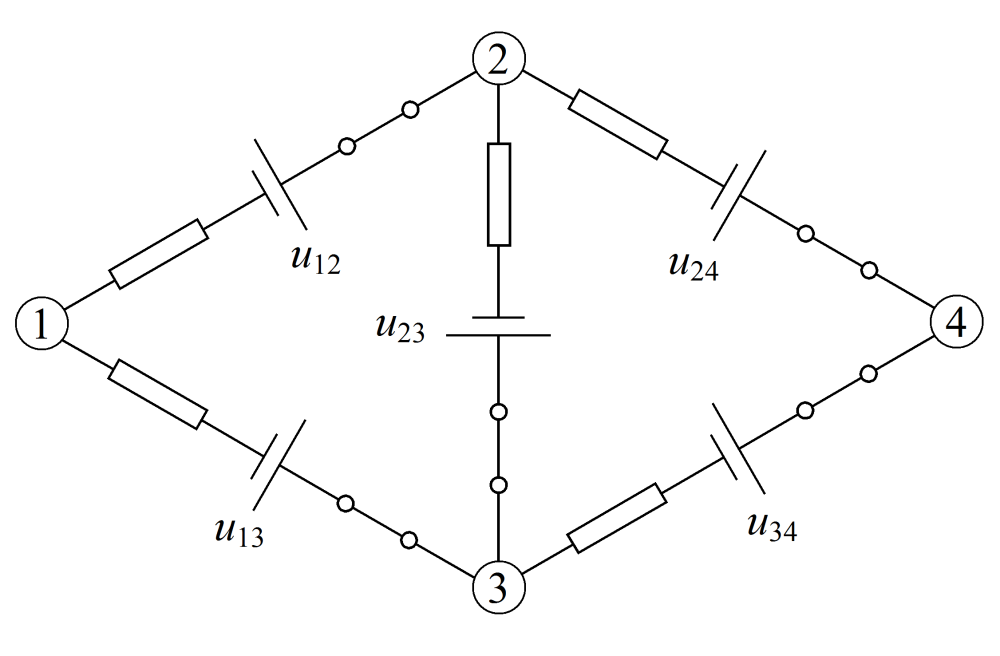}}}}
\put(5,0){\makebox{\emph{Figure 4. The electric circuit on four nodes in Example \ref{circuitexample}}}}
\end{picture}
\end{center}

Every resistor has the same resistance $R$.
The values of $u_{12},u_{13},u_{23},u_{24},u_{34}$ are arbitrary real numbers.
The aim is to calculate the potentials $U_1,U_2,U_3,U_4$ of nodes 1,2,3,4 such that the total energy (power) of the system is minimal.
  The objective function follows from a physical law by nature.
 The total energy is the sum of electrical powers ($V \cdot I = \frac{V^2}{R}$) of the resistors,
 where $V$ denotes the potential dif{\kern0pt}ference (voltage drop) across the given resistor and
 $I$ denotes the current through it. For a resistor between nodes $i$ and $j$,
 $V = u_{ij} - U_{i} + U_{j}$. Since resistance $R$ is assumed to be constant,
 the objective function to be minimized is the sum (for all edges $(i,j)$ in the graph) of terms
 $\left(u_{ij} - U_{i} + U_{j} \right)^2$.
 We have the optimization problem (4) %(\ref{eq:IncompleteLSMProblem-ObjFunction})
with the incomplete  additive   (skew symmetric) matrix
\[
\mathbf{B} =
\begin{pmatrix}
    0       &    u_{12}      &     u_{13}     &                 \\
 - u_{12}   &        0       &     u_{23}     &      u_{24}     \\
 - u_{13}   &  - u_{23}      &         0      &      u_{34}     \\
            &  - u_{24}      &   - u_{34}     &         0
\end{pmatrix}
\]
and   variables   $\mathbf{y} = (U_1=0,U_2,U_3,U_4)^{\top}.$
 It is worth noting that if (and only if) matrix $\mathbf{B}$ is consistent according to
Def{\kern0pt}inition 1.2, then currents are zeros and $U^{\ast}_i - U^{\ast}_j = u_{ij}$ for all edges $(i,j)$,
 the total power   of the circuit is zero.

Assume two loop currents $I_a$ and $I_b$ around loops 1231 and 2432 and write
Kirchhof{\kern0pt}f's Voltage Law
  (the directed sum of the potential dif{\kern0pt}ferences around any closed loop is zero,
 (compare to Def{\kern0pt}inition 1.2)):
\begin{align}
RI_a + u_{12} + R(I_a-I_b) + u_{23} - u_{13} + RI_a = 0 \nonumber \\
RI_b + u_{24} - u_{34} + RI_b - u_{23} + R(I_b-I_a) = 0 \nonumber
\end{align}
that results in
\begin{align}
I_a = \frac{-3u_{12}+3u_{13}-2u_{23}-u_{24}+u_{34}}{8R} \nonumber \\
I_b = \frac{-u_{12}+u_{13}+2u_{23}-3u_{24}+3u_{34}}{8R}. \nonumber
\end{align}
Assume without loss of generality that $U_1 = 0.$
Then
\begin{align}
U_2 &= U_1 + RI_a + u_{12} = \frac{5}{8}u_{12}+\frac{3}{8}u_{13}-\frac{1}{4}u_{23}-\frac{1}{8}u_{24}+\frac{1}{8}u_{34} \nonumber \\
U_3 &= U_1 -RI_a + u_{13} =  \frac{3}{8}u_{12}+\frac{5}{8}u_{13}+\frac{1}{4}u_{23}+\frac{1}{8}u_{24}-\frac{1}{8}u_{34}
 \label{eq:KirchhoffsVoltageCircuitLaws} \\ %\ref{eq:KirchhoffsVoltageCircuitLaws}
U_4 &= U_2 + RI_b + u_{24} = \frac{1}{2}u_{12}+\frac{1}{2}u_{13}+\frac{1}{2}u_{24}+\frac{1}{2}u_{34} \nonumber
\end{align}

  Kirchhof{\kern0pt}f's Current Law (the signed sum of currents is zero for every node) can
      be also verif{\kern0pt}ied. \\

Now let us consider the spanning tree approach.
Graph $G$ has 8 spanning trees shown in Figure 5, the corresponding circuits are given in Figure 6.

\unitlength 1mm
\begin{center}
\begin{picture}(100,70)
\put(0,10){\resizebox{100mm}{!}{\rotatebox{0}{
\includegraphics{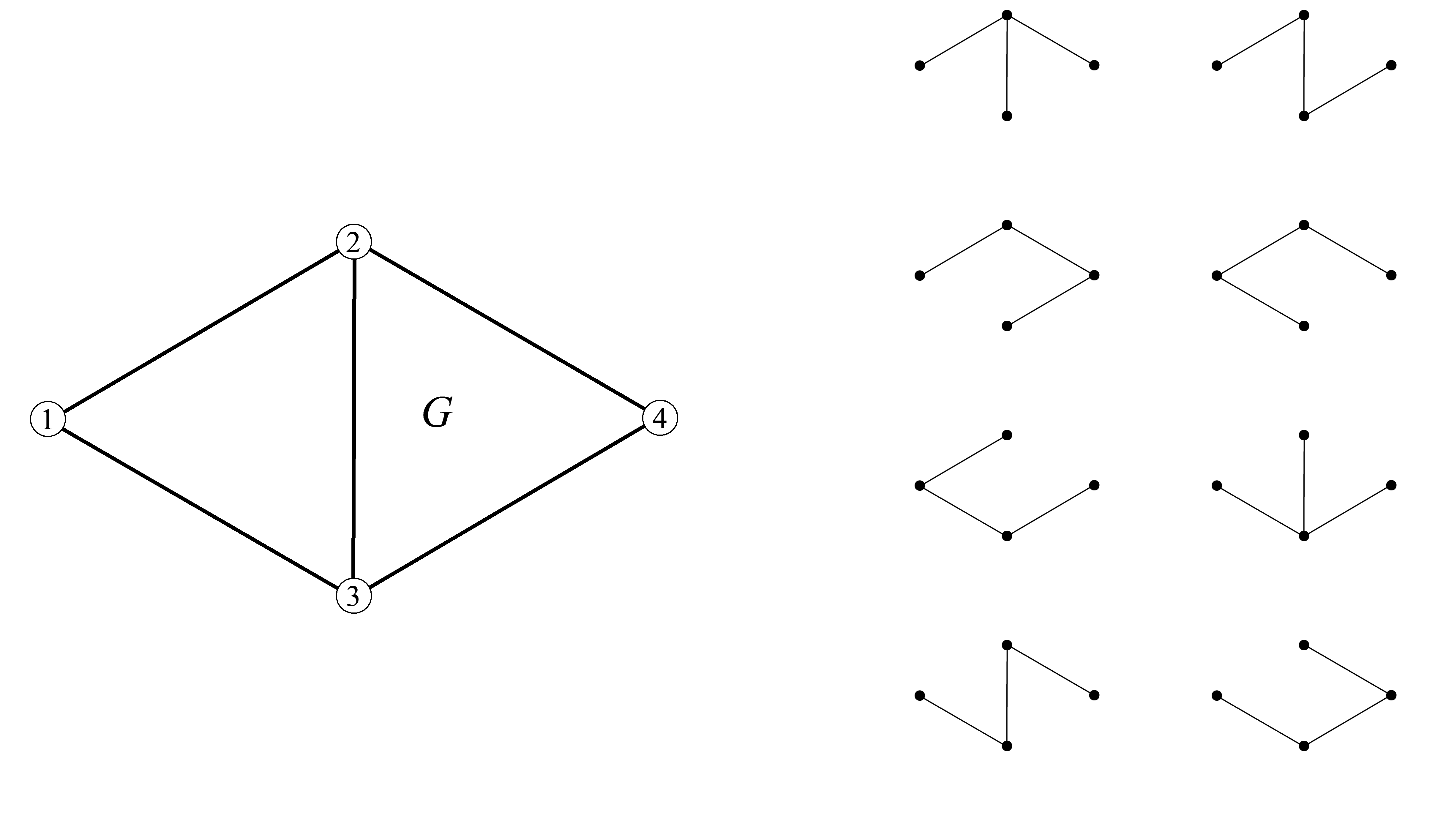}}}}
\put(10,5){\makebox{\emph{Figure 5. Graph $G$ of Example
\ref{circuitexample} and its 8 spanning trees}}}
\end{picture}
\end{center}

\newpage
\unitlength 1mm
\begin{center}
\begin{picture}(100,200)
\put(-25,00){\resizebox{150mm}{!}{\rotatebox{0}{
\includegraphics{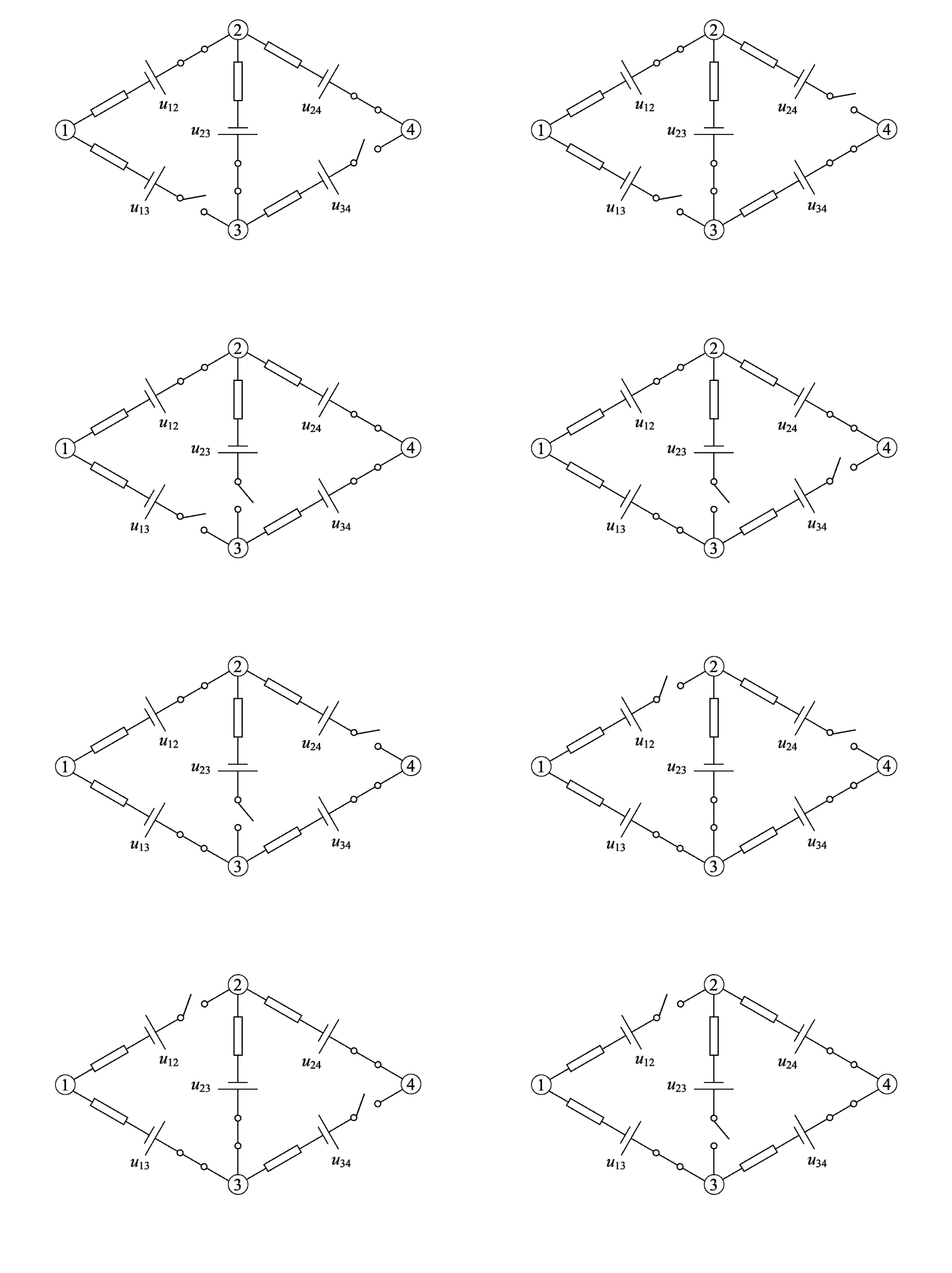}}}}
\put(-10,0){\makebox{\emph{Figure 6. Circuits corresponding to the
8 spanning trees of Example \ref{circuitexample}}}}
\end{picture}
\end{center}

We shall apply Theorem \ref{LSMtheorem}, without loss of generality  we assume again that $U_1 = 0$.
The calculation of the potentials is elementary for every spanning tree, because the (signed) voltages along the unique
path from node 1 to another node are summed: \\

\begin{tabular}{|c||c|c|c|c|}
\hline
spanning tree    &    $U_1$   &            $U_2$           &          $U_3$           &         $U_4$             \\
\hline
\hline
\begin{picture}(10,7)
\put(-2,-2){\resizebox{13mm}{!}{\rotatebox{0}{\includegraphics{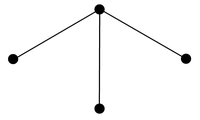}}}}
\end{picture}
                 &     0    &            $u_{12}$        &   $u_{12}+u_{23}$         &     $u_{12} + u_{24}$     \\[2mm]
\hline
\begin{picture}(10,7)
\put(-2,-2){\resizebox{13mm}{!}{\rotatebox{0}{\includegraphics{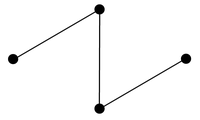}}}}
\end{picture}
                &     0    &            $u_{12}$        &   $u_{12}+u_{23}$         & $u_{12}+u_{23}+u_{34}$     \\[2mm]
\hline
\begin{picture}(10,7)
\put(-2,-2){\resizebox{13mm}{!}{\rotatebox{0}{\includegraphics{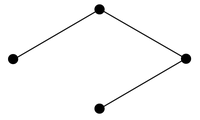}}}}
\end{picture}
               &     0    &            $u_{12}$        & $u_{12}+u_{24}-u_{34}$     &     $u_{12} + u_{24}$      \\[2mm]
\hline
\begin{picture}(10,7)
\put(-2,-2){\resizebox{13mm}{!}{\rotatebox{0}{\includegraphics{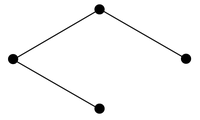}}}}
\end{picture}
                &     0    &            $u_{12}$        &           $u_{13}$       &     $u_{12} + u_{24}$      \\[2mm]
\hline
\begin{picture}(10,7)
\put(-2,-2){\resizebox{13mm}{!}{\rotatebox{0}{\includegraphics{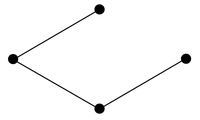}}}}
\end{picture}
                &     0    &            $u_{12}$        &           $u_{13}$       &     $u_{13} + u_{34}$      \\[2mm]
\hline
\begin{picture}(10,7)
\put(-2,-2){\resizebox{13mm}{!}{\rotatebox{0}{\includegraphics{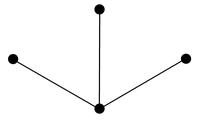}}}}
\end{picture}
                &     0    &     $u_{13} - u_{23}$      &           $u_{13}$       &     $u_{13} + u_{34}$      \\[2mm]
\hline
\begin{picture}(10,7)
\put(-2,-2){\resizebox{13mm}{!}{\rotatebox{0}{\includegraphics{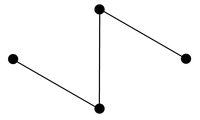}}}}
\end{picture}
                &     0    &     $u_{13} - u_{23}$      &           $u_{13}$       & $u_{13}-u_{23}+ u_{24}$    \\[2mm]
\hline
\begin{picture}(10,7)
\put(-2,-2){\resizebox{13mm}{!}{\rotatebox{0}{\includegraphics{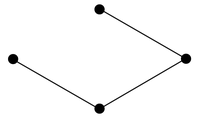}}}}
\end{picture}
                &     0    & $u_{13}+u_{34}-u_{24}$     &           $u_{13}$       &     $u_{13} + u_{34}$      \\[2mm]
\hline
\hline
                &          &                            &                           &                            \\
arithmetic mean &     0    &
  $\frac{5}{8}u_{12}+\frac{3}{8}u_{13}-\frac{1}{4}u_{23}$ &
                             $\frac{3}{8}u_{12}+\frac{5}{8}u_{13}+\frac{1}{4}u_{23}$ &
                                                                           $\frac{1}{2}u_{12}+\frac{1}{2}u_{13}$  \\
                 &          &
                   $-\frac{1}{8}u_{24}+\frac{1}{8}u_{34}$ &
                                              $+\frac{1}{8}u_{24}-\frac{1}{8}u_{34}$ &
                                                                           $ +\frac{1}{2}u_{24}+\frac{1}{2}u_{34}$  \\
                 &          &                            &                          &                              \\
\hline
\end{tabular}

\begin{center}
\emph{Table 1. Potentials calculated from the 8 spanning trees of Example \ref{circuitexample} } \\[3mm]
\end{center}

The arithmetic means in Table 1 are the same as the ones derived from Kirchhof{\kern0pt}f's laws
given in (\ref{eq:KirchhoffsVoltageCircuitLaws}).

According to Theorem \ref{LSMtheorem} the arithmetic means in Table 1 satisfy the following system (in an
analogous way to (\ref{equationLaplacian})-(\ref{normalizationy1=0})):
\[
\begin{pmatrix}
    2       &      - 1       &      - 1       &        0          \\
   - 1      &        3       &      - 1       &      - 1        \\
   - 1      &      - 1       &        3       &      - 1        \\
    0       &      - 1       &      - 1       &        2
\end{pmatrix}
\begin{pmatrix}
 0 \\
 \frac{5}{8}u_{12}+\frac{3}{8}u_{13}-\frac{1}{4}u_{23}-\frac{1}{8}u_{24}+\frac{1}{8}u_{34}   \\
 \frac{3}{8}u_{12}+\frac{5}{8}u_{13}+\frac{1}{4}u_{23}+\frac{1}{8}u_{24}-\frac{1}{8}u_{34}   \\
 \frac{1}{2}u_{12}+\frac{1}{2}u_{13}+\frac{1}{2}u_{24}+\frac{1}{2}u_{34}
\end{pmatrix}
=
\begin{pmatrix}
      u_{12} + u_{13}           \\
 - u_{12} + u_{23} + u_{24}     \\
 - u_{13} - u_{23} + u_{34}     \\
 - u_{24}          - u_{34}
\end{pmatrix},
\]
where the matrix above is the Laplacian of $G$, and the right hand side is the vector
of row elements' sum in $\mathbf{B}$.
\end{example}

\section{Conclusions}
It was shown in this paper that two weighting methods,
based on rather dif{\kern0pt}ferent principles and approaches,
 are equivalent not only for complete pairwise comparison matrices,
as it was recently proved by Lundy, Siraj and Greco
\cite{LundySirajGreco2017}, but also for incomplete ones.
The  arithmetic (geometric)
 mean of weight vectors calculated from all spanning
trees was proved to be (logarithmic)  least squares optimal.
 The proof of the complete case
\cite{LundySirajGreco2017} cannot be extended to the incomplete
case, due to that the incomplete (L)LS optimal solution does not
have an explicit formula. However, the implicit formula
(\ref{equationLaplacian}) was still applicable to operations with
spanning trees.

The advantages rooted in the def{\kern0pt}inition of the two methods,
namely the clear interpretation of taking all spanning trees into account and
the optimality by a widely analyzed objective functions (LLS, LS), are now united.
 Spanning trees not only unfold the graph of comparisons,
 but their corresponding weight vectors also provide an expressive decomposition of the
 (logarithmic) least squares optimal weight vector.
 An important consequence of the paper is that future analyses of weighting methods should not
distinguish between the incomplete LLS/LS       and the
geometric/arithmetic        mean of weight vectors from all spanning trees.

There is a signif{\kern0pt}icant dif{\kern0pt}ference in
computational complexity. The (logarithmic)   least squares problem
can be solved from a single system of linear equations (the
coef{\kern0pt}f{\kern0pt}icient matrix is the Laplacian),
requiring at most  $O(n^{2.376})$ steps in theory
\cite{Spielman2010}. However, recent approximate and iterative
algorithms optimized for large and
suf{\kern0pt}f{\kern0pt}iciently sparse matrices run in nearly
linear time \cite{Spielman2010,Vishnoi2013}. The enumeration of
all spanning trees with the algorithm of Gabow and Myers
\cite{GabowMyers1978}, requires $O(n+m+nS)$ steps, where $m$
denotes the number of edges in $G$. The computational complexity
of calculating all weight vectors, associated to the spanning
trees, is $\max\{O(nS), O(n+m+nS) \}$ steps, where $S$, the number
of spanning trees, is between 1 and $n^{n-2}$. We can conclude
that, except for special matrices whose associated graph has a
small number of spanning trees, the (logarithmic) least squares
problem is faster to solve.

Certain applications apply the spanning trees enumeration, but not necessarily
together with the aggregation by the geometric mean.
The approach of spanning trees enumeration is used in determining the consistency
to build the distribution of expert estimates based on the matrix \cite{OlenkoTsyganok2016}.
Such problems of{\kern0pt}fer further research possibilities.

The possible equivalence of  some
%the arithmetic (or other but not geometric)
mean of weight vectors, calculated from all spanning trees and other weighting methods, is still an open problem.

Taking weights into consideration in (logarithmic) least squares problem
(see, e.g., \cite{Barzilai1997} and \cite[Chapter 6]{Lootsma1999})
 is a possible extension. In group decision making, weights represent the voting powers of the individual decision makers.
 Multiple comparisons for the same pairs, or considering information quality and source credibility
 also lead to weighted models with objective functions
$ \sum v_{ij} \left[\log a_{ij}-\log\left(\frac{w_{i}}{w_{j}}\right)\right]^2 $
or
$ \sum v_{ij} \left(b_{ij} - y_{i} + y_{j} \right)^2. $
 An extension of Theorems \ref{maintheorem} and \ref{LSMtheorem} to the weighted case is more than inspiring.
Note that the weighted variant of the corresponding representation with electric
circuits and potentials in Section 3 leads to non-identical resistances.

\section*{Acknowledgements}
The constructive remarks of the anonymous reviewers are greatly acknowledged.
The authors would like to
show their gratitude to Satoru Fujishige (Research Institute for
Mathematical Sciences, Kyoto University) for his remark, on the
analogy with electric circuits, that he made at the 10th
Japanese-Hungarian Symposium on Discrete Mathematics and Its
Applications, May 22-25, 2017, Budapest, Hungary.
 The authors are grateful to Andr\'as Recski
 (Budapest University of Technology and Economics)
for his substantial comments.
J\'anos F\"ul\"op (Institute for Computer Science and Control, Hungarian
Academy of Sciences (MTA SZTAKI) and \'Obuda University, Budapest)
is greatly acknowledged for his valuable comments on the
computational complexity of solving the Laplacian equation.
 Orsolya Csisz\'ar is greatly acknowledged for her careful proofreading.
S.~Boz\'oki acknowledges the support of the J\'anos Bolyai Research Fellowship
 of the Hungarian Academy of Sciences (no.~BO/00154/16/3);
 the \'{U}NKP-18-4-BCE-90 Bolyai+ New National Excellence Program of the Ministry of Human Capacities, Hungary;
and the Hungarian Scientif{\kern0pt}ic Research Fund (OTKA), grant no.~K111797.


\begin{thebibliography}{99}


\bibitem{Barzilai1997}
Barzilai, J. (1997): Deriving weights from pairwise comparison
matrices, Journal of the Operational Research Society
\textbf{48}(12) 1226--1232
% DOI 10.1057/palgrave.jors.2600474
%
% https://link.springer.com/article/10.1057/palgrave.jors.2600474

\bibitem{Barzilai1998}
Barzilai, J. (1998):
Consistency measures for pairwise comparison matrices,
Journal of Multi-Criteria Decision Analysis
\textbf{7}(3) 123--132
% DOI 10.1002/(SICI)1099-1360(199805)7:3<123::AID-MCDA181>3.0.CO;2-8
%
% https://onlinelibrary.wiley.com/doi/abs/10.1002/(SICI)1099-1360(199805)7:3%3C123::AID-MCDA181%3E3.0.CO;2-8

\bibitem{BarzilaiCookGolany1987}
Barzilai, J., Cook, W.D., Golany, B. (1987):
Consistent weights for judgements matrices of the relative importance of alternatives,
Operations Research Letters
\textbf{6}(3) 131--134
% DOI 10.1016/0167-6377(87)90026-5
% https://www.sciencedirect.com/science/article/pii/0167637787900265
% https://web.iem.technion.ac.il/images/user-files/golany/papers/ORL_87.pdf

\bibitem{BarzilaiGolany1990}
Barzilai, J., Golany, B. (1990):
Deriving weights from pairwise comparison matrices: The additive case,
Operations Research Letters
\textbf{9}(6) 407--410
% DOI 10.1016/0167-6377(90)90062-A
% https://www.sciencedirect.com/science/article/pii/016763779090062A
%

\bibitem{BenitezDelgado-GalvanIzquierdoPerez-Garcia2015}
% J.Benítez X.Delgado-Galván J.Izquierdo R.Pérez-García
% Ben\'{\i}tez, Delgado-Galv\'{a}n, Izquierdo nad P\'{e}rez-Garc\'{\i}a
Ben\'{\i}tez, J., Delgado-Galv\'{a}n, X., Izquierdo, J., P\'{e}rez-Garc\'{\i}a, R. (2015):
Consistent completion of incomplete judgments in decision making using AHP,
Journal of Computational and Applied Mathematics
\textbf{290}(15) 412--422
% DOI 10.1016/j.cam.2015.05.023
% https://www.sciencedirect.com/science/article/pii/S0377042715003179?via%3Dihub

\bibitem{BozokiCsatoTemesi2016}
Boz\'oki, S., Csat\'o, L., Temesi, J. (2016):
An application of incomplete pairwise comparison matrices for ranking top tennis players,
European Journal of Operational Research
\textbf{248}(1) 211--218
% DOI 10.1016/j.ejor.2015.06.069
% http://www.sciencedirect.com/science/article/pii/S0377221715006220

\bibitem{BozokiFulopRonyai2010}
Boz\'oki, S., F\"ul\"op, J., R\'onyai, L. (2010):
On optimal completion of incomplete pairwise comparison matrices,
Mathematical and Computer Modelling
 \textbf{52}(1-2) 318--333
% DOI 10.1016/j.mcm.2010.02.047
% http://www.sciencedirect.com/science/article/pii/S0895717710001159

\bibitem{Brugha2000}
Brugha, C.M. (2000):
Relative measurement and the power function,
European Journal of Operational Research
\textbf{121}(3) 627--640
% DOI 10.1016/S0377-2217(99)00057-0
% https://www.sciencedirect.com/science/article/pii/S0377221799000570

\bibitem{Brunelli2016}
Brunelli, M. (2016):
A technical note on two inconsistency indices for preference relations: A case of functional relation,
Information Sciences
\textbf{357}, 1--5
% DOI 10.1016/j.ins.2016.03.048
% https://www.sciencedirect.com/science/article/pii/S0020025516302237

\bibitem{Brunelli2015}
Brunelli, M. (2015): Introduction to the Analytic Hierarchy Process,
Springer, Cham
% ISBN: 978-3-319-12501-5 (print) 978-3-319-12502-2 (online)
% ISSN: 2195-0482
% DOI 10.1007/978-3-319-12502-2
% http://link.springer.com/book/10.1007%2F978-3-319-12502-2
% https://books.google.hu/books?id=u0DVBQAAQBAJ

\bibitem{Brunelli2018}
Brunelli, M. (2018):
A survey of inconsistency indices for pairwise comparisons,
International Journal of General Systems
 \textbf{47}(8) 751--771
% DOI 10.1080/03081079.2018.1523156
% https://www.tandfonline.com/doi/abs/10.1080/03081079.2018.1523156

\bibitem{CaklovicKurdija2017}
% Lavoslav Caklovic, Adrian Satja Kurdija
\v{C}aklovi\'{c}, L., Kurdija, A.S. (2017):
A universal voting system based on the Potential Method,
European Journal of Operational Research
\textbf{259}(2) 677--688
% DOI 10.1016/j.ejor.2016.10.032
%https://www.sciencedirect.com/science/article/pii/S0377221716308669

\bibitem{CarmoneKaraZanakis1997}
Carmone, F., Kara, A., Zanakis, S.H. (1997):
A Monte Carlo investigation of incomplete pairwise comparison matrices in AHP,
European Journal of Operational Research
\textbf{102}(3) 538--553.
% DOI 10.1016/S0377-2217(96)00250-0
% https://doi.org/10.1016/S0377-2217(96)00250-0
% https://www.sciencedirect.com/science/article/pii/S0377221796002500

\bibitem{ChaoKouLiPeng2018}
Chao, X., Kou, G., Li, T., Peng, Y. (2018):
Jie Ke versus AlphaGo:
A ranking approach using decision making method for large-scale data with incomplete information,
European Journal of Operational Research
 \textbf{265}(1) 239--247
% DOI 10.1016/j.ejor.2017.07.030
%
% https://www.sciencedirect.com/science/article/pii/S0377221717306574

\bibitem{ChebotarevShamis1999}
Chebotarev, P.Y., Shamis, E. (1999):
Preference fusion when the number of alternatives exceeds two: indirect scoring procedures,
Journal of the Franklin Institute
\textbf{336}(2) 205--226
% DOI 10.1016/S0016-0032(98)00017-9
%Pavel Yu. Chebotarev, Elena Shamis
% https://www.sciencedirect.com/science/article/pii/S0016003298000179

\bibitem{ChooWedley2004}
Choo, E.U., Wedley, W.C. (2004):
A common framework for deriving preference values from pairwise comparison matrices,
Computers \& Operations Research
\textbf{31}(6) 893--908
% DOI 10.1016/S0305-0548(03)00042-X
% http://www.sciencedirect.com/science/article/pii/S030505480300042X

\bibitem{ChuKalabaSpingarn1979}
Chu, A.T.W., Kalaba, R.E., Spingarn, K. (1979):
A comparison of two methods for determining the weights of belonging to fuzzy sets,
Journal of Optimization Theory and Applications
\textbf{27}(4) 531--538
% DOI 10.1007/BF00933438
% http://link.springer.com/article/10.1007/BF00933438

\bibitem{CrawfordWilliams1985}
Crawford, G., Williams, C. (1985):
A note on the analysis of subjective judgment matrices,
Journal of Mathematical Psychology
\textbf{29}(4) 387--405
% DOI 10.1016/0022-2496(85)90002-1
% http://www.sciencedirect.com/science/article/pii/0022249685900021

\bibitem{Csato2013}
Csat\'o, L. (2013):
Ranking by pairwise comparisons for Swiss-system tournaments,
Central European Journal of Operations Research
\textbf{21}(4) 783--803
% DOI 10.1007/s10100-012-0261-8
% http://link.springer.com/article/10.1007/s10100-012-0261-8

\bibitem{Csato2019}
Csat\'o, L. (2019):
A characterization of the Logarithmic Least Squares Method,
European Journal of Operational Research,
\textbf{276}(1) 212--216
% DOI 10.1016/j.ejor.2018.12.046
% https://www.sciencedirect.com/science/article/pii/S0377221718311202
% https://arxiv.org/abs/1704.05321

\bibitem{DulebaMishinaShimazaki2012}
Duleba, S., Mishina, T., Shimazaki, Y. (2012):
A dynamic analysis on public bus transport's supply quality by using AHP,
Transport
\textbf{27}(3) 268--275
% DOI 10.3846/16484142.2012.719838
% http://dx.doi.org/10.3846/16484142.2012.719838

\bibitem{Edwards1977}
Edwards, W. (1977):
How to use multiattribute utility measurement for social decision making,
IEEE Transactions on Systems, Man, and Cybernetics
\textbf{7}(5) 326--340
% DOI 10.1109/TSMC.1977.4309720
% https://ieeexplore.ieee.org/abstract/document/4309720/

\bibitem{FedrizziGiove2007}
Fedrizzi, M., Giove, S. (2007):
Incomplete pairwise comparison and consistency optimization,
European Journal of Operational Research
\textbf{183}(1) 303--313
% DOI 10.1016/j.ejor.2006.09.065
% https://doi.org/10.1016/j.ejor.2006.09.065
% https://www.sciencedirect.com/science/article/pii/S0377221706010022

\bibitem{Fichtner1984}
Fichtner, J. (1984):
Some thoughts about the Mathematics of the Analytic Hierarchy Process.
Report 8403,
Universit\"at der Bundeswehr M\"unchen,
Fakult\"at f\"ur Informatik,
Institut f\"ur Angewandte Systemforschung und Operations Research,
Werner-Heisenberg-Weg 39, D-8014 Neubiberg, F.R.G.
1984.

\bibitem{Fichtner1986}
Fichtner, J. (1986):
On deriving priority vectors from matrices of pairwise comparisons.
Socio-Economic Planning Sciences
\textbf{20}(6) 341--345
% DOI 10.1016/0038-0121(86)90045-5
% http://www.sciencedirect.com/science/article/pii/0038012186900455


\bibitem{GabowMyers1978}
% Harold N. Gabow and Eugene W. Myers
Gabow, H.N., Myers, E.W. (1978):
Finding all spanning trees of directed and undirected graphs,
SIAM Journal on Computing
 \textbf{7}(3) 280--287
% DOI 10.1137/0207024
% http://dx.doi.org/10.1137/0207024
%  http://epubs.siam.org/doi/abs/10.1137/0207024

\bibitem{Gass1998}
Gass, S.I. (1998):
Tournaments, transitivity and pairwise comparison matrices,
Journal of the Operational Research Society
 \textbf{49}(6) 616--624
% DOI 10.1057/palgrave.jors.2600572
% https://www.tandfonline.com/doi/abs/10.1057/palgrave.jors.2600572

\bibitem{GolanyKress1993}
Golany, B., Kress, M. (1993):
A multicriteria evaluation of methods for obtaining weights from ratio-scale matrices,
European Journal of Operational Research
\textbf{69}(2) 210--220
% DOI 10.1016/0377-2217(93)90165-J
% http://www.sciencedirect.com/science/article/pii/037722179390165J

\bibitem{deGraan1980}
de Graan, J.G. (1980):
Extensions of the multiple criteria analysis method of T.L.~Saaty.
Technical Report m.f.a. 80-3, National Institute for Water Supply, Leidschendam, The Netherlands.
Presented at EURO IV, Cambridge, July 22-25, 1980.

\bibitem{GrecoEhrgottFigueira2016}
Greco, S., Ehrgott, M., Figueira, J.R. (Eds.):
Multiple Criteria Decision Analysis: State of the Art Surveys, 2nd  edition,
International Series in Operations Research and Management Science, Volume 233,
Springer, 2016
% DOI 10.1007/978-1-4939-3094-4
% https://www.springer.com/la/book/9781493930937
% ISBN 978-1-4939-3093-7

\bibitem{Harker1987b}
Harker, P.T. (1987):
Incomplete pairwise comparisons in the analytic hierarchy process,
Mathematical Modelling
\textbf{9}(11) 837--848
% DOI 10.1016/0270-0255(87)90503-3
% http://www.sciencedirect.com/science/article/pii/0270025587905033

\bibitem{HarkerVargas1987}
Harker, P.T., Vargas, L.G. (1988):
The theory of ratio scale estimation: Saaty's Analytic Hierarchy Process,
Management Science
\textbf{33}(11) 1367--1509
% DOI 10.1287/mnsc.33.11.1383
% https://pubsonline.informs.org/doi/abs/10.1287/mnsc.33.11.1383

\bibitem{Ho2008}
Ho, W. (2008):
Integrated analytic hierarchy process and its applications -- A literature review,
European Journal of Operational Research
\textbf{186}(1) 211--228
% DOI 10.1016/j.ejor.2007.01.004
% https://www.sciencedirect.com/science/article/pii/S0377221707000872

\bibitem{Horst1932}
Horst, P. (1932):
A method for determining the absolute af{\kern0pt}fective value of a series of stimulus situations,
Journal of Educational Psychology,
\textbf{23}(6) 418--440
% DOI 10.1037/h0070134
% http://dx.doi.org/10.1037/h0070134
% http://psycnet.apa.org/record/1933-00054-001

\bibitem{deJong1984}
de Jong, P. (1984):
A statistical approach to Saaty's scaling methods for priorities,
Journal of Mathematical Psychology
\textbf{28}(4) 467--478
% DOI 10.1016/0022-2496(84)90013-0
% http://www.sciencedirect.com/science/article/pii/0022249684900130

\bibitem{Krejci2018}
% Jana Krejèí
Krej\v{c}i, J.:
Pairwise Comparison Matrices and their Fuzzy Extension -- Multi-Criteria Decision Making
with a new Fuzzy Approach, Springer, 2018
% DOI 10.1007/978-3-319-77715-3
% ISSN 1434-9922
% ISSN 1860-0808 (electronic)
% Studies in Fuzziness and Soft Computing
% ISBN 978-3-319-77714-6
% ISBN 978-3-319-77715-3 (eBook)
% https://doi.org/10.1007/978-3-319-77715-3


\bibitem{Kwiesielewicz1996}
Kwiesielewicz, M. (1996):
The logarithmic least squares and the generalised pseudoinverse in estimating ratios,
European Journal of Operational Research
\textbf{93}(3) 611--619
% DOI 10.1016/0377-2217(95)00079-8
% http://www.sciencedirect.com/science/article/pii/0377221795000798

\bibitem{Lootsma1999}
Lootsma, F.A. (1999):
Multi-Criteria Decision Analysis via Ratio and Dif{\kern0pt}ference Judgement,
Series of Applied Optimization, Volume 29, Kluwer, Dordrecht/Boston/London
% DOI 10.1007/b102374
% https://www.springer.com/us/book/9780792356691
% Hardcover ISBN 978-0-7923-5669-1
% Softcover ISBN 978-1-4419-4810-6
% eBook ISBN 978-0-585-28008-0
% Series ISSN 1384-6485

\bibitem{LundySirajGreco2017}
Lundy, M., Siraj, S., Greco, S. (2017):
The mathematical equivalence of the ``spanning tree''
and row geometric mean preference vectors and its implications for preference analysis,
European Journal of Operational Research
\textbf{257}(1) 197--208
% DOI 10.1016/j.ejor.2016.07.042
% http://www.sciencedirect.com/science/article/pii/S0377221716305975

\bibitem{MengChen2015}
Meng, F., Chen, X. (2105):
An approach to incomplete multiplicative preference relations and its application in group decision making,
Information Sciences
\textbf{309}, 119--137
% DOI 10.1016/j.ejor.2016.07.042
% https://www.sciencedirect.com/science/article/pii/S0020025515001851?via%3Dihub

\bibitem{OlenkoTsyganok2016}
Olenko, A., Tsyganok, V. (2016):
Double Entropy Inter-Rater Agreement Indices,
Applied Psychological Measurement
\textbf{40}(1) 37--55
% doi: 10.1177/0146621615592718.
% http://journals.sagepub.com/doi/abs/10.1177/0146621615592718

\bibitem{OlivaScalaSetolaDellOlmo2019}
% Gabriele Oliva, Antonio Scala, Roberto Setola, Paolo Dell'Olmo
Oliva, G., Scala, A., Setola, R., Dell'Olmo, P. (2019):
Opinion-Based Optimal Group Formation,
Omega, DOI 10.1016/j.omega.2018.10.008
% DOI 10.1016/j.omega.2018.10.008
% https://doi.org/10.1016/j.omega.2018.10.008
% https://www.sciencedirect.com/science/article/pii/S0305048318306479

\bibitem{OlsonFliednerCurrie1995}
Olson, D.L., Fliedner, G., Currie, K. (1995):
Comparison of the REMBRANDT system with analytic hierarchy process,
European Journal of Operational Research
\textbf{82}(3) 522--539
% DOI
% https://www.sciencedirect.com/science/article/pii/0377221793E03404

\bibitem{Rabinowitz1976}
Rabinowitz, G. (1976):
Some comments on measuring world inf{\kern0pt}luence,
Journal of Peace Science
\textbf{2}(1) 49--55
% DOI 10.1177/073889427600200104
% https://doi.org/10.1177/073889427600200104
% http://journals.sagepub.com/doi/abs/10.1177/073889427600200104?journalCode=cmpa

\bibitem{Saaty1977}
Saaty, T.L. (1977):
A scaling method for priorities in hierarchical structures,
Journal of Mathematical Psychology
\textbf{15}(3) 234--281
% doi:10.1016/0022-2496(77)90033-5
% http://www.sciencedirect.com/science/article/pii/0022249677900335

\bibitem{ShiraishiObataDaigo1998}
Shiraishi, S., Obata, T., Daigo, M. (1998):
Properties of a positive reciprocal matrix and their application to AHP,
Journal of the Operations Research Society of Japan
\textbf{41}(3) 404--414
% DOI 10.15807/jorsj.41.404
% https://doi.org/10.15807/jorsj.41.404
% https://www.jstage.jst.go.jp/article/jorsj/41/3/41_KJ00001201864/_article/-char/ja/

\bibitem{SirajMikhailovKeane2012a}
Siraj, S., Mikhailov, L., Keane, J.A. (2012):
Enumerating all spanning trees for pairwise comparisons,
Computers \& Operations Research
\textbf{39}(2) 191--199
% doi:10.1016/j.cor.2011.03.010
% http://www.sciencedirect.com/science/article/pii/S0305054811000839

\bibitem{SirajMikhailovKeane2012b}
Siraj, S., Mikhailov, L., Keane, J.A. (2012):
Corrigendum to ``Enumerating all spanning trees for pairwise comparisons [Comput.~Oper.~Res.~39(2012) 191-199]'',
Computers \& Operations Research
\textbf{39}(9) page 2265
% doi:10.1016/j.cor.2011.11.010
% http://www.sciencedirect.com/science/article/pii/S0305054811003352

\bibitem{Spielman2010}
Spielman, D.A. (2010):
Algorithms, graph theory, and linear equations in Laplacian matrices,
Proceedings of the International Congress of Mathematicians,
Hyderabad, India, 2010, pages 2698--2722
% DOI 10.1.1.165.8870 ?
% http://dx.doi.org/10.1.1.165.8870
% http://www.cs.yale.edu/homes/spielman/PAPERS/icm10post.pdf
% http://citeseerx.ist.psu.edu/viewdoc/download?doi=10.1.1.165.8870&rep=rep1&type=pdf
% https://pdfs.semanticscholar.org/22ae/6f9b9c9984fd00f7446693fbcddfb5142f3b.pdf

\bibitem{SubramanianRamanathan2012}
Subramanian, N., Ramanathan, R. (2012):
A review of applications of Analytic Hierarchy Process in operations management
International Journal of Production Economics
\textbf{138}(2) 215--241
% DOI 10.1016/j.ijpe.2012.03.036
% http://dx.doi.org/10.1016/j.ijpe.2012.03.036
% https://www.sciencedirect.com/science/article/pii/S0925527312001442?via%3Dihub

\bibitem{TakedaYu1995}
Takeda, E., Yu,  P.L. (1995):
Assessing priority weights from subsets of pairwise comparisons in multiple criteria optimization problems,
European Journal of Operational Research
\textbf{86}(2) 315--331.
% DOI 10.1016/0377-2217(95)00095-8
% https://doi.org/10.1016/0377-2217(95)00095-8
% https://www.sciencedirect.com/science/article/pii/0377221795000958

\bibitem{Thurstone1927}
Thurstone, L.L. (1927): Psychophysical Analysis,
American Journal of Psychology
\textbf{38}(3) 368--389
% DOI 10.2307/1415006
% https://www.jstor.org/stable/1415006

\bibitem{Tsyganok2000}
Tsyganok, V. (2000):
Combinatorial method of pairwise comparisons with feedback,
Data Recording, Storage \& Processing
\textbf{2}, 92--102 (in Ukrainian).

\bibitem{Tsyganok2010}
Tsyganok, V. (2010):
Investigation of the aggregation ef{\kern0pt}fectiveness of expert estimates obtained by the pairwise comparison method,
Mathematical and Computer Modelling
\textbf{52}(3-4) 538--54
% doi: 10.1016/j.mcm.2010.03.052
% http://www.sciencedirect.com/science/article/pii/S0895717710001706

\bibitem{TsyganokKadenkoAndriichuk2015}
Tsyganok, V.V., Kadenko, S.V., Andriichuk, O.V. (2015):
Using dif{\kern0pt}ferent pair-wise comparison scales for developing industrial strategies,
International Journal of Management and Decision Making
\textbf{14}(3) 224--250
% DOI 10.1504/IJMDM.2015.070760
% http://www.inderscienceonline.com/doi/abs/10.1504/IJMDM.2015.070760

\bibitem{UrenaChiclanaMorente-MolineraHerrera-Viedma2015}
Ure\~{n}a, R., Chiclana, F., Morente-Molinera, J.A., Herrera-Viedma, E. (2015):
Managing incomplete preference relations in decision making: A review and future trends,
Information Sciences
\textbf{302}, 14--32
% DOI 10.1016/j.ins.2014.12.061
% https://www.sciencedirect.com/science/article/pii/S0020025515000134
% https://doi.org/10.1016/j.ins.2014.12.061

\bibitem{VaidyaKumar2006}
Vaidya, O.S., Kumar, S. (2006):
Analytic hierarchy process: An overview of applications,
European Journal of Operational Research
\textbf{169}(1) 1--29
% DOI 10.1016/j.ejor.2004.04.028
% https://doi.org/10.1016/j.ejor.2004.04.028
% https://www.sciencedirect.com/science/article/pii/S0377221704003054

\bibitem{Vishnoi2013}
Vishnoi, N.K. (2013):
$Lx=b$ Laplacian solvers and their algorithmic applications,
Foundations and Trends in Theoretical Computer Science
\textbf{8}(1-2) 1--141
% DOI: 10.1561/0400000054
% http://theory.epfl.ch/vishnoi/Lxb-Web.pdf

%\bibitem{Watkins2002}
%Watkins, D.S. (2002):
%Fundamentals of Matrix Computations,
%Second Edition. John Wiley \& Sons, Inc., New York
%% ISBN 0-471-21394-2


\end{thebibliography}
\end{document}